\documentclass[a4paper,10pt]{amsart}

\usepackage[utf8]{inputenc}

\usepackage{graphicx}

\usepackage{cite}
\usepackage{enumerate}

\usepackage{hyperref}

\hypersetup{%
pdftitle={},
pdfsubject={Mathematics},
pdfauthor={Manfred Madritsch},
pdfkeywords={}
hyperindex=true,plainpages=false}

\usepackage{a4wide}

\usepackage{amsmath}
\usepackage{amsfonts}
\usepackage{amssymb}
\usepackage{amsthm}

\newtheorem{lem}{Lemma}[section]
\newtheorem{thm}[lem]{Theorem}
\newtheorem{prop}[lem]{Proposition}
\newtheorem{cor}[lem]{Corollary}

\numberwithin{equation}{section}

\newtheorem*{cor*}{Corollary}
\newtheorem*{thm*}{Theorem}

\theoremstyle{definition}
\newtheorem{defi}{Definition}[section]

\theoremstyle{remark}

\allowdisplaybreaks[3]

\newcommand{\N}{\mathbb{N}}
\newcommand{\Z}{\mathbb{Z}}
\newcommand{\Q}{\mathbb{Q}}
\newcommand{\R}{\mathbb{R}}

\newcommand{\T}{\mathbb{T}}

\renewcommand{\lvert}{\left\vert}
\renewcommand{\rvert}{\right\vert}

\title{Dynamical systems and uniform distribution of sequences}

\author[M. G. Madritsch]{Manfred G. Madritsch}
\address[M. G. Madritsch]{
\noindent 1. Universit\'e de Lorraine, Institut Elie Cartan de Lorraine, UMR 7502, Vandoeuvre-l\`es-Nancy, F-54506, France;\newline
\noindent 2. CNRS, Institut Elie Cartan de Lorraine, UMR 7502, Vandoeuvre-l\`es-Nancy, F-54506, France}
\email{manfred.madritsch@univ-lorraine.fr}

\author[R. F. Tichy]{Robert F. Tichy}
\address[R. F. Tichy]{Department for Analysis and Computational Number
  Theory\\Graz University of Technology\\A-8010 Graz, Austria}
\email{tichy@tugraz.at}

\dedicatory{Dedicated to the memory of Professor Wolfgang Schwarz}

\subjclass[2010]{}


\date{\today}

\begin{document}

\begin{abstract}
  We give a survey on classical and recent applications of dynamical
  systems to number theoretic problems. In particular, we focus on
  normal numbers, also including computational aspects. The main
  result is a sufficient condition for establishing multidimensional
  van der Corput sets. This condition is applied to various examples.
\end{abstract}

\maketitle

\section{Dynamical systems in number theory}
In the last decades dynamical systems became very important for the
development of modern number theory. The present paper focuses on
Furstenberg's refinements of Poincaré's recurrence theorem and
applications of these ideas to Diophantine problems.

A (measure-theoretic) dynamical system is formally given as a
quadruple $(X,\mathfrak{B}, \mu, T)$, where $(X,\mathfrak{B},\mu)$
is a probability space with $\sigma$-algebra $\mathfrak{B}$ of
measurable sets and $\mu$ a probability measure; $T\colon
X\rightarrow X$ is a measure-preserving transformation on this
space, \textit{i.e.} $\mu(T^{-1}A)= \mu(A)$ for all measurable
sets $A\in\mathfrak{B}$. In the theory of dynamical systems,
properties of the iterations of the transformation $T$ are of
particular interest. For this purpose we only consider invertible
transformations and call such dynamical systems invertible.

The first property, we consider, originates from Poncaré's famous
recurrence theorem (see Theorem 1.4 of
\cite{walters1982:introduction_to_ergodic} or Theorem 2.11 of
\cite{einsiedler_ward2011:ergodic_theory_with}) saying that
starting from a set $A$ of positive measure $\mu(A)>0$ and
iterating $T$ yields infinitely many returns to $A$. More
generally, we call a subset $\mathcal{R}\subset\mathbb{N}$ of the
positive integers a set of recurrence if for all invertible
dynamical systems and all measurable sets $A$ of positive measure
$\mu (A)>0$ there exists $n\in\mathcal{R}$ such that $\mu(A\cap
T^{-n}A)>0$. Then Poincaré's recurrence theorem means that $\N$
is a set of recurrence.

A second important theorem for dynamical systems is Birkhoff's ergodic
theorem (see Theorem 1.14 of
\cite{walters1982:introduction_to_ergodic} or Theorem 2.30 of
\cite{einsiedler_ward2011:ergodic_theory_with}). We call $T$ ergodic
if the only invariant sets under $T$ are sets of measure $0$ or of
measure $1$, \textit{i.e.} $T^{-1}A=A$ implies $\mu(A)=0$ or
$\mu(A)=1$. Then Birkhoff's ergodic theorem connects average in time
with average in space, \textit{i.e.}
\[\lim_{N\rightarrow\infty}\frac{1}{N}\sum^{N-1}_{n=0}f \circ
T^{n}(x)= \int_{X}f(x)d\mu(x)\]
for all $f\in L^{1}(X,\mu)$ and $\mu$-almost all $x \in X$.

Let us explain an important application of this theorem to number
theory. For $q\geq2$ a positive integer, consider
$T\colon[0,1)\rightarrow [0,1)$ defined by $T(x)=\{qx\},$ where
$\{t\}=t-\lfloor t\rfloor$ denotes the fractional part of $t$.
If $x\in\R$ is given by its $q$-ary digit expansion
$x=\lfloor x\rfloor+ \sum^{\infty}_{j=1}a_{j}(x)q^{-j}$, then the
digits $a_{j}(x)$ can be computed by iterating this transformation
$T$: $a_{j}(x)=i$ if $T^{j-1}x\in\left[\tfrac iq,\frac{i+1}q\right)$
with $i\in\{0,1,\ldots,q-1\}$. Moreover, since $a_j(Tx)=a_{j+1}(x)$
for $j\geq1$ the transformation $T$ can be seen as a left shift of the
expansion.

Now we call a real number $x$ simply normal in base $q$ if
\[\lim_{\mathbb{N}\rightarrow\infty}\frac{1}{N}\#\{j\leq N\colon a_{j}=d\}= \frac{1}{q}\]
for all $d=0, \ldots, q-1$, \textit{i.e.} all digits $d$ appear
asymptotically with equal frequencies $1/q.$ A number $x$ is
called $q$-normal if it is simply normal with respect to all bases
$q, q^{2},q^{3},\ldots$. This is equivalent to the fact that the
sequence $(\{q^{n}x\})_{n\in\mathbb{N}}$ is uniformly distributed
modulo $1$ (for short: u.d.~mod $1$), which also means that all
blocks $d_{1},d_{2},\ldots, d_{L}$ of subsequent digits appear in
the expansion of $x$ asymptotically with the same frequency
$q^{-L}$ (\textit{cf.} \cite{bugeaud2012:distribution_modulo_one,
  drmota_tichy1997:sequences_discrepancies_and,
  kuipers_niederreiter1974:uniform_distribution_sequences}). For
completeness, let us give here one possible definition of u.d.
sequences $(x_{n})$: a sequence of real numbers $x_{n}$ is called
u.d.~mod $1$ if for all continuous functions $f: [0,1]\rightarrow
\mathbb{R}$

\begin{equation}\label{ud}
\lim_{N\rightarrow\infty}\frac{1}{N}\sum^{N}_{n=1}f(x_{n})=
\int^{1}_{0}f(x)dx.
\end{equation}

Note, that by Weyl's criterion the class of continuous functions
can be replaced by trigonometric functions $e(hx)=e^{2\pi ihx}$,
$h\in \mathbb{N}$ or by characteristic functions $1_{I}(x)$ of
intervals $I=[a,b)$. Applying Birkhoff's ergodic theorem, shows
that Lebesgue almost all real numbers are $q$-normal in any base
$q\geq 2$. Defining a real number to be absolutely normal if it is
$q$-normal for all bases $q\geq 2$, this immediately yields that
almost all real numbers are absolutely normal.

In particular, this shows the existence of absolutely normal
numbers. However, it is a different story to find constructions of
(absolutely) normal numbers. It is a well-known difficult open problem
to show that important numbers like $\sqrt{2}$, $\ln 2$, $e$, $\pi$
etc. are simply normal with respect to some given base $q\geq 2$. A
much easier task is to give constructions of $q$-normal numbers for
fixed base $q$. Champernowne
\cite{champernowne1933:construction_decimals_normal} proved that

$$0.1\,2\,3\,4\,5\,6\,7\,8\,9\,10\,11\,12\ldots$$

is normal to base $10$ and later this type of constructions was
analysed in detail. So, for instance, for arbitrary base $q\geq 2$

\[0.\langle\lfloor g(1)\rfloor\rangle_{q}\; \langle\lfloor g(2)\rfloor\rangle_{q}\ldots\]

is $q$-normal, where $g(x)$ is a non-constant polynomial with real
coefficients and the $q-$normal number is constructed by concatenating
the $q-$ary digit expansions $\langle\lfloor g(n)\rfloor\rangle_{q}$
of the integer parts of the values $g(n)$ for $n=1,2,\ldots$. These
constructions were extended to more general classes of functions $g$
(replacing the polynomials) (see
\cite{nakai_shiokawa1992:discrepancy_estimates_class,
  nakai_shiokawa1990:class_normal_numbers,
  madritsch_thuswaldner_tichy2008:normality_numbers_generated,
  madritsch_tichy2013:construction_normal_numbers,
  davenport_erdoes1952:note_on_normal,
  schiffer1986:discrepancy_normal_numbers}) and the concatenation of
$\langle[g(p)]\rangle_{q}$ along prime numbers instead of the positive
integers (see \cite{nakai_shiokawa1997:normality_numbers_generated,
  madritsch2014:construction_normal_numbers,
  copeland_erdoes1946:note_on_normal,
  madritsch_tichy2013:construction_normal_numbers}).

All such constructions depend on the choice of the base number
$q\geq 2$, and thus they are not suitable for constructing
absolutely normal numbers. A first attempt to construct absolutely
normal numbers is due to Sierpinski
\cite{sierpinski1917:demonstration_elementaire_du}. However,
Turing \cite{turing1992:note_on_normal} observed that Sierpinski's
``construction'' does not yield a computable number, thus it is
not based on a recursive algorithm. Furthermore, Turing gave an
algorithm for a construction of an absolutely normal number. This
algorithm is very slow and, in particular, not polynomially in
time. It is very remarkable that Becher \text{et al.}
\cite{becher_heiber_slaman2013:polynomial_time_algorithm}
established a polynomial time algorithm for the construction of
absolutely normal numbers. However, there remain various questions
concerning the analysis of these algorithms. The discrepancy of
the corresponding sequences is not studied and the order of
convergence of the expansion is very slow and should be
investigated in detail. Furthermore, digital expansions with
respect to linear recurring base sequences seam appropriate to be
included in the study of absolute normality from a computational
point of view.

Let us now return to Poincaré's recurrence theorem which shortly states
that the set $\mathbb{N}$ of positive integers is a recurrence set. In
the 1960s various stronger concepts were introduced:
\begin{enumerate}[(i)]
\item $\mathcal{R}\subseteq \mathbb{N}$ is called a nice recurrence
  set if for all invertible dynamical systems and all measurable sets
  $A$ of positive measure $\mu(A)>0$ and all $\varepsilon > 0,$ there
  exist infinitely many $n\in \mathcal{R}$ such that $$\mu(A\cup
  T^{-n}A)>\mu(A)^{2}- \varepsilon.$$
\item $\mathcal{H}\subseteq\mathbb{N}$ is called a van der Corput set
  (for short: vdC set) if for all $h\in\mathcal{H}$ the following
  implications holds:
  \[(x_{n+h}-x_{n})_{n\in\mathbb{N}}\; \text{is u.d.~mod $1$}\Longrightarrow
  (x_{n})_{n\in\mathbb{N}}\; \text{is u.d.~mod $1$.}\]
\end{enumerate}

Clearly, any nice recurrence set is a recurrence set. By van der
Corput's difference theorem (see
\cite{kuipers_niederreiter1974:uniform_distribution_sequences,
  drmota_tichy1997:sequences_discrepancies_and}) the set
$\mathcal{H}= \N$ of positive integers is a vdC set. Kamae and
Mend\`es-France \cite{kamae_mendes1978:van_der_corputs} proved that
any vdC set is a nice recurrence set. Ruzsa
\cite{ruzsa1984:connections_between_uniform} conjectured that any
recurrence set is also vdC. An important tool in the analysis of
recurrence sets is their equivalence with intersective (or difference)
sets established by Bertrand-Mathis
\cite{bertrand-mathis1986:ensembles_intersectifs_et}. We call a set
$\mathcal{I}$ intersective if for each subset $E\subseteq \mathbb{N}$
of positive (upper) density, there exists $n\in\mathcal{I}$ such that
$n=x-y$ for some $x,y\in E$. Here the upper density of $E$ is defined
as usual by
\[\overline{d}(E)=\limsup_{N\to\infty}\frac{\#(E\cap[1,N])}{N}.\]
Bourgain \cite{bourgain1987:ruzsas_problem_on} gave an example of an
intersective set which is not a vdC set, hence contradicting the above
mentioned conjecture of Ruzsa.

Furstenberg \cite{furstenberg1977:ergodic_behavior_diagonal}
proved that the values $g(n)$ of a polynomial $g\in \mathbb{Z}[x]$
with $g(0)=0$ form an intersective set and later it was shown by
Kamae and Mend\`es-France \cite{kamae_mendes1978:van_der_corputs}
that this is a vdC set, too. It is also known, that for fixed
$h\in \mathbb{Z}$ the set of shifted primes $\{p\pm h\colon
p\text{ prime}\}$ is a vdC set if and only if $h=\pm 1.$
(\cite[Corollary 10]{montgomery1994:ten_lectures_on}). This leads
to interesting applications to additive number theory, for
instance to new proofs and variants of theorems of S\'ark\H{o}zy
\cite{sarkozy1978:difference_sets_sequences1,
  sarkoezy1978:difference_sets_sequences3,
  sarkoezy1978:difference_sets_sequences2}. A general result
concerning intersective sets related to polynomials along primes is
due to Nair \cite{nair1992:certain_solutions_diophantine}.

In the present paper we want to extend the concept of recurrence sets,
nice recurrence sets and vdC sets to subsets of $\Z^{k},$ following
the program of Bergelson and Lesigne
\cite{bergelson_lesigne2008:van_der_corput} and our earlier
paper~\cite{bergelson_kolesnik_madritsch+2014:uniform_distribution_prime}. In
section 2 we summarize basic facts concerning these concepts,
including general relations between them and counter examples. Section
3 is devoted to a sufficient condition for establishing the vdC
property. In the final section 4 we collect various examples and give
some new applications.

\section{Van der Corput sets}

In this section we provide various equivalent definitions of van der
Corput sets in $\Z^k$. In particular, we give four different
definitions, which are $k$-dimensional variants of the one dimensional
definitions, whose equivalence is due to Ruzsa
\cite{ruzsa1984:connections_between_uniform}. These generalizations
were established by Bergelson and Lesigne
\cite{bergelson_lesigne2008:van_der_corput}.  Then we present a set,
which is not a vdC set in order to give some insight into the
structure of vdC sets. Finally, we define the
higher-dimensional variant of nice recurrence sets.

\subsection{Characterization via uniform distribution}

Similarly to above we first define a van der Corput set (vdC set for
short) in $\Z^k$ via
uniform distribution.
\begin{defi}
  A subset $\mathcal{H}\subset\Z^k\setminus\{0\}$ is a vdC set if any
  family $(x_{\mathbf{n}})_{\mathbf{n}\in\N^k}$ of real numbers is u.d.~mod $1$
  provided that it has
  the property that for all $\mathbf{h}\in \mathcal{H}$ the family
  $(x_{\mathbf{n}+\mathbf{h}}-x_{\mathbf{n}})_{\mathbf{n}\in\N^k}$ is
  u.d.~mod $1$.

  Here the property of u.d.~mod $1$ for the multi-indexed family
  $(x_{\mathbf{n}})_{\mathbf{n}\in\N^k}$ is defined via a natural
  extension of \ref{ud}:
\end{defi}

\begin{equation}\label{ud1}
\lim_{N_1,N_2,\ldots,N_k\to+\infty}\frac1{N_1N_2\cdots N_k}
  \sum_{0\leq\mathbf{n}<(N_1,N_2,\ldots,N_k)}f(x_{\mathbf n})=
\int^{1}_{0}f(x)dx
\end{equation}

for all continuous functions $f:[0,1]\rightarrow\mathbb{R}.$ Here
in the limit $N_1, N_2,\ldots, N_k$ are tending to infinity
independently and $<$ is defined componentwise.

Using the $k$-dimensional variant of van der Corput's inequality
we could equivalently define a vdC set as follows:

\begin{defi}
  A subset $\mathcal{H}\subset\Z^k\setminus\{0\}$ is a van der Corput
  set if for any family $(u_\mathbf n)_{n\in\Z^k}$ of complex numbers of
  modulus $1$ such that \[\forall\mathbf{h}\in \mathcal{H},\quad
  \lim_{N_1,N_2,\ldots,N_k\to+\infty}\frac1{N_1N_2\cdots N_k}
  \sum_{0\leq\mathbf{n}<(N_1,N_2,\ldots,N_k)}u_{\mathbf{n}+\mathbf{h}}\overline{u_{\mathbf{n}}}=0\]
the relation
\[\lim_{N_1,N_2,\ldots,N_k\to+\infty}\frac1{N_1N_2\cdots
    N_k}
  \sum_{0\leq\mathbf{n}<(N_1,N_2,\ldots,N_k)}u_{\mathbf{n}}=0\]
  holds.
\end{defi}

\subsection{Trigonometric polynomials and spectral characterization}

The first two definitions are not very useful for proving or
disproving that a set $\mathcal{H}$ is a vdC set. Similar to the one
dimensional case the following spectral characterization involving
trigonometric polynomials is a better tool.

\begin{thm}[{\cite[Proposition
    1.18]{bergelson_lesigne2008:van_der_corput}}]
  A subset $\mathcal{H}\subset\Z^k\setminus\{0\}$ is a van der Corput set if and
  only if for all $\varepsilon>0$, there exists a real trigonometric
  polynomial $P$ on the $k$-torus $\T^k$ whose spectrum is contained
  in $\mathcal{H}$ and which satisfies $P(0)=1$, $P\geq-\varepsilon$.
\end{thm}

The set of polynomials fulfilling the last theorem for a given
$\varepsilon$ forms a convex set. Moreover the conditions may be
interpreted as some infimum. Therefore we might expect some dual
problem, which is actually provided by the following theorem. For
details see Bergelson and Lesigne
\cite{bergelson_lesigne2008:van_der_corput} or Montgomery
\cite{montgomery1994:ten_lectures_on}.

\begin{thm}[{\cite[Theorem
    1.8]{bergelson_lesigne2008:van_der_corput}}]
  Let $\mathcal{H}\subset\Z^k\setminus\{0\}$. Then $\mathcal{H}$ is a
  van der Corput set if and only if for any positive measure $\sigma$
  on the $k$-torus $\T^k$ such that, for all $\mathbf{h}\in \mathcal{H}$,
  $\widehat{\sigma}(\mathbf{h})=0$, this implies
  $\sigma(\{(0,0,\ldots,0)\})=0$.
\end{thm}

\subsection{Examples}

The structure of vdC sets is better understood by first giving a
counter example. The following lemma shows to be very useful in
the construction of counter examples.

\begin{lem}\label{mt:infinite_intersection}
Let $\mathcal{H}\subset\N$. If there exists $q\in\N$ such that the
set $\mathcal{H}\cap q\N$ is finite, then the set $\mathcal{H}$ is
not a vdC set.
\end{lem}

\begin{proof}
The proof is a combination of the following two observations of Ruzsa
\cite{ruzsa1984:connections_between_uniform} (see Theorem 2
and Corollary 3 of \cite{montgomery1994:ten_lectures_on}):
\begin{enumerate}
\item Let $m\in\N$. The sets $\{1,\ldots,m\}$ and $\{n\in\N\colon
  m\nmid n\}$ are both not vdC sets.
\item Let $\mathcal{H}=\mathcal{H}_1\cup\mathcal{H}_2\subset\N$. If
  $\mathcal{H}$ is a vdC set, then $\mathcal{H}_1$ or $\mathcal{H}_2$
  also has to be a vdC set.
\end{enumerate}
Suppose there exists a $q\in\N$ such that $\mathcal{H}\cap q\N$ is
finite. Then we may split $\mathcal{H}$ into the sets
$\mathcal{H}\cap q\N$ and $\mathcal{H}\setminus q\N$. The first
one is finite and the second one contains no multiples of $q$.
Therefore both are not vdC sets and hence $\mathcal{H}$ is not a
vdC set.
\end{proof}

The first counter example deals with arithmetic progressions.
\begin{lem}\label{lem:arithmetic_prog}
  Let $a,b\in\N$. If the set $\{an+b\colon n\in\N\}$ is a vdC set,
  then $a\mid b$.
\end{lem}

\begin{proof}
Let $b\in\N$ and $\mathcal{H}=\{an+b\colon n\in\N\}$ be a
vdC set. Then by Lemma \ref{mt:infinite_intersection} we must have
\[an+b\equiv b\equiv 0\bmod a\quad\text{infinitely often.}\]
This implies that $a\mid b$.
\end{proof}

The sufficiency (and also the necessity) of the requirement $a\mid b$
follows from the following result of Kamae and Mend\`es-France
\cite{kamae_mendes1978:van_der_corputs} (\textit{cf.} Corollary 9 of
\cite{montgomery1994:ten_lectures_on}).
\begin{lem}
Let $P(z)\in\Z[z]$ and suppose that $P(z)\to+\infty$ as
$z\to+\infty$. Then $\mathcal{H}=\{P(n)>0\colon n\in\N\}$ is a vdC set
if and only if for every positive integer $q$ the congruence
$P(z)\equiv 0\pmod q$ has a root.
\end{lem}

Now we want to establish a similar result for sets of the form
$\{ap+b\colon p\text{ prime}\}$. In this case the following result is
due to Bergelson and Lesigne
\cite{bergelson_lesigne2008:van_der_corput} which is a generalization of
the case $f(x)=x$ due to Kamae and Mend\`es-France
\cite{kamae_mendes1978:van_der_corputs}.
\begin{lem}[{\cite[Proposition
    1.22]{bergelson_lesigne2008:van_der_corput}}]\label{bl:prop1.22}
Let $f$ be a (non zero) polynomial with integer coefficients and
zero constant term. Then the sets $\{f(p-1)\colon
p\in\mathbb{P}\}$ and $\{f(p+1)\colon p\in\mathbb{P}\}$ are vdC
sets in $\Z$.
\end{lem}

We show the converse direction.
\begin{lem}
Let $a$ and $b$ be non-zero integers. Then the set $\{ap+b\colon
p\in\mathbb{P}\}$ is a vdC set if and only if $\lvert a\rvert=\lvert
b\rvert$, \textit{i.e.} $ap+b=a(p\pm 1)$.
\end{lem}

\begin{proof}
It is clear from Lemma \ref{bl:prop1.22} that $\{ap+b\colon
p\in\mathbb{P}\}$ is a vdC set if $\lvert a\rvert=\lvert b\rvert$.

On the contrary a combination of Lemma \ref{mt:infinite_intersection}
and Lemma \ref{lem:arithmetic_prog} yields that $a\mid b$. Now we
consider the sequence modulo $b$. Then by Lemma
\ref{mt:infinite_intersection} we get that
\[ap+b\equiv ap\equiv 0\bmod b\quad\text{infinitely often.}\] Since $(p,b)>1$ only
holds for finitely many primes $p$ we must have $b\mid a$. Combining
these two requirements yields $\lvert a\rvert=\lvert b\rvert$.
\end{proof}

\section{A sufficient condition}

In this section we want to formulate a general sufficient
condition which provides us with a tool to show for plenty of
different examples that they generate a vdC set. This is a
generalization of the conditions of Kamae and Mend\`es-France
\cite{kamae_mendes1978:van_der_corputs} and Bergelson and Lesigne
\cite{bergelson_lesigne2008:van_der_corput}. Before stating the
condition we need an auxiliary lemma.

\begin{lem}[{\cite[Corollary
    1.15]{bergelson_lesigne2008:van_der_corput}}]
\label{bl:lem_linear_algebra}
  Let $d$ and $e$ be positive integers, and let $L$ be a linear
  transformation from $\Z^d$ into $\Z^e$ (represented by an $e\times
  d$ matrix with integer entries). Then the following assertions hold:
\begin{enumerate}
\item If $D$ is a vdC set in $\Z^d$ and if $0\not\in L(D)$, then
  $L(D)$ is a vdC set in $\Z^e$.
\item Let $D\in\Z^d$. If the linear map $L$ is one-to-one, and if
  $L(D)$ is a vdC set in $\Z^e$, then $D$ is a vdC set in $\Z^d$.
\end{enumerate}
\end{lem}

Our main tool is the following general result. Applications
are given in the next section.

\begin{prop}\label{mt:sufficient_condition}
Let $g_1,\ldots,g_k\colon\N\to\Z$ be arithmetic functions. Suppose
that $g_{i_1},\ldots,g_{i_m}$ is a basis of the $\Q$-vector space
$\mathrm{span}(g_1,\ldots,g_k)$. For each $q\in\N$, we introduce
\[D_q:=\left\{(g_{i_1}(n),\ldots,g_{i_m}(n)\colon n\in\N\text{ and $q!\mid
  g_{i_j}(n)$ for all $j=1,\ldots,m$}\right\}.\]
Suppose further that, for every $q$, there exists a sequence
$(h^{(q)}_n)_{n\in\N}$ in $D_q$ such that, for all
$\mathbf{x}=(x_1,\ldots,x_m)\in\R^m\setminus\Q^m$, the sequence
$(h^{(q)}_n\cdot\mathbf{x})_{n\in\N}$ is uniformly distributed mod
$1$. Then
\[\widetilde{D}:=\{(g_1(n),\ldots,g_k(n))\colon n\in\N\}\in\Z^k\]
is a vdC set.
\end{prop}

\begin{proof}
We first show that the set
\[D:=\{(g_{i_1}(n),\ldots,g_{i_m}(n))\colon n\in\N\}\]
is a vdC set in $\Z^m$. For $q,N\in\N$ we define a family of trigonometric
polynomials
\[P_{q,N}:=\frac1N\sum_{n=1}^Ne\left(h_n^{(q)}\cdot\mathbf{x}\right).\]
By hypothesis, $\lim_{N\to\infty}P_{q,N}(x)=0$ for
$x\not\in\Q^m$. For fixed $q$ there exists a subsequence $(P_{q,N'})$
which converges pointwise to a function $g_q$. Since $g_q(x)=1$
(for $x\in\Q^m$ and $q$ sufficiently large) and $g_q(x)=0$ (for
$x\not\in\Q^m$), the sequence $(g_q)$ is pointwise convergent to the indicator
function of $\Q^m$. For a positive measure $\sigma$ on the
$m$-dimensional torus with vanishing Fourier transform
$\widehat{\sigma}$ on $D$, we have $\int P_{q,N}\mathrm{d}\sigma=0$
for all $q,N$. Thus $\sigma(\Q^m)=0$ follows from the dominating
convergence theorem, obviously $\sigma(\{0,0,\ldots,0\})=0$, and thus $D$ is a vdC
set.

In order to prove that $\widetilde{D}$ is a vdC set we apply Lemma
\ref{bl:lem_linear_algebra} twice. Since $g_{i_1},\ldots,g_{i_m}$ is a
base of $\mathrm{span}(g_1,\ldots,g_k)$, we can write each $g_j$ as a
linear combination (with rational coefficients) of
$g_{i_1},\ldots,g_{i_m}$. Multiplying with the common denominator of
the coefficients yields
\[a_jg_j=b_{j,1}g_{i_1}+\cdots+b_{j,m}g_{i_m}\]
for $j=1,\ldots,k$ and certain $a_j,b_{j,\ell}\in\Z$. Considering the
transformation $L\colon\Z^m\to\Z^k$ given by the matrix $(b_{j,\ell})$
and applying part (1) of Lemma \ref{bl:lem_linear_algebra} shows that
\[\{(a_1g_1(n),\ldots,a_kg_k(n))\colon n\in\N\}\]
is a vdC set for certain integers $a_1,\ldots,a_k$.

Now consider the transformation $\widetilde{L}\colon\Z^k\to\Z^k$ given by
the $k\times k$ diagonal matrix with entries $a_1,\ldots,a_k$ in the
diagonal. Then by part (2) of Lemma \ref{bl:lem_linear_algebra} also
$\widetilde{D}$ is a vdC set and the proposition is proved.
\end{proof}

\section{Various examples and applications to additive problems}

In this section we consider multidimensional variants of prime
powers, entire functions and $x^\alpha\log^\beta x$ sequences.

\subsection{Prime powers}

In a recent paper the authors together with Bergelson, Kolesnik and
Son
\cite{bergelson_kolesnik_madritsch+2014:uniform_distribution_prime}
consider sets of the form
\[\{(\alpha_1(p_n\pm1)^{\theta_1},\ldots,\alpha_k(p_n\pm1)^{\theta_k})\colon
n\in\N\},\] where $\alpha_i,\beta_i\in\R$ and $p_n\in\mathcal{P} $
runs over all prime numbers. These sets are vdC, however, we
missed the treatment of a special case in the proof. In
particular, if for some $i\neq j$ the exponents satisfy
$\theta_i=\theta_j=:\theta$, then the vector
$(p_n^{\theta},p_n^{\theta})$ is not uniformly distributed mod 1.

Here we close this gap.
\begin{thm}
If $\alpha_i$ are positive integers and $\beta_i$ are positive and non-integers, then
$$ D_1 = \{ \left(  (p-1)^{\alpha_1}, \cdots , (p-1)^{\alpha_k}, [(p-1)^{\beta_1}], \cdots , [(p-1)^{\beta_\ell}]  \right) | \, p \in\mathcal{P} \},$$
and
$$ D_2 = \{ \left(  (p+1)^{\alpha_1}, \cdots , (p+1)^{\alpha_k}, [(p+1)^{\beta_1}], \cdots , [(p+1)^{\beta_\ell}]  \right) | \, p \in \mathcal{P} \}$$
are vdC sets in $\Z^{k+\ell}$.
\end{thm}

\begin{proof}
Since $x^{\theta_1}$ and $x^{\theta_2}$ are $\Q$-linear dependent
for all $x\in\Z$ if and only if $\theta_1=\theta_2$, an
application of Proposition \ref{mt:sufficient_condition} yields
that it suffices to consider the case where all exponents are
different. However, this follows by the same arguments as in the
proof of Theorem 4.1 in
\cite{bergelson_kolesnik_madritsch+2014:uniform_distribution_prime}.
\end{proof}

\subsection{Entire functions}

In this section we consider entire functions of bounded logarithmic
order. We fix a transcendental entire function $f$ and denote by
$S(r):=\max_{\lvert z\rvert\leq r}\lvert f(z)\rvert$. Then we call
$\lambda$ the logarithmic order of $f$ if
\[\limsup_{r\to\infty}\frac{\log S(r)}{\log r}=\lambda.\]

The central tool is the following result of Baker
\cite{baker1984:entire_functions_and}.

\begin{thm}[{\cite[Theorem 2]{baker1984:entire_functions_and}}]\label{baker:uniform_distribution_of_entire_functions}
Let $f$ be a transcendental entire function of logarithmic order
$1<\lambda<\frac43$. Then the sequence
\[\left(f(p_n)\right)_{n\geq1}\]
is uniformly distributed mod $1$.
\end{thm}

Our second example of a class of vdC sets is the following.
\begin{thm}\label{thm:entire_functions_vdC}
Let $f_1,\ldots,f_k$ be entire functions with distinct logarithmic orders
$1<\lambda_1,\lambda_2,\ldots,\lambda_k<\frac43$, respectively.
Then the set
\[D:=\{(\lfloor f_1(p_n)\rfloor,\ldots,\lfloor f_k(p_n)\rfloor)\colon
n\in\N\}\]
is a vdC set.
\end{thm}

\begin{proof}
  We enumerate $D=(\mathbf{d}_n)_{n\geq1}$, where
  \[\mathbf{d}_n:=\left(\lfloor f_1(p_n)\rfloor,\ldots,\lfloor f_k(p_n)\rfloor\right).\]

  First we show, that for every $q\in\N$ the set
  \[D^{(q)}:=\{(d_1,\ldots,d_k)\in D\colon q\mid d_i\}\] has positive
  relative density in $D$. We note that if
  $0\leq\left\{\frac{f_i(p_n)}{q}\right\}<\frac1q$ for $1\leq i\leq k$,
  then $\mathbf{d}_n\in D^{(q)}$. By Theorem
  \ref{baker:uniform_distribution_of_entire_functions} the
  sequence \[\left(\left(\frac{f_1(p_n)}{q},\ldots,\frac{f_k(p_n)}q\right)\right)_{n\geq1}\]
  is uniformly distributed and thus $D^{(q)}$ has positive density in
  $D$.

  For each $q\in\N$ we enumerate the elements of
  $D^{(q!)}=(\mathbf{d}^{(q!)}_n)_{n\geq1}$, such that
  $\lvert\mathbf{d}^{(q!)}_n\rvert$ is increasing. Since the
  logarithmic orders are distinct we immediately get that the
  functions $f_i$ are $\Q$-linearly independent. Thus by Proposition
  \ref{mt:sufficient_condition} it is sufficient to show that for all
  $q\in\N$ and all $\mathbf{x}=(x_1,\ldots,x_k)\in\R^k\setminus\Q^k$
  the sequence $(\mathbf{d}^{(q!)}_n\cdot\mathbf{x})_{n\geq1}$ is
  u.d. mod 1.

  Using the orthogonality relations for additive characters we get for
  any non-zero integer $h$, that
  \begin{multline*}
  \frac1{\lvert\{n\leq N\colon\mathbf{d}_n\in D^{(q!)}\}\rvert}
    \sum_{n\leq N}e\left(h\left(d^{(q!)}_n\cdot\mathbf{x}\right)\right)\\
  =\frac1{\lvert\{n\leq N\colon\mathbf{d}_n\in D^{(q!)}\}\rvert}
    \frac1{(q!)^k}\sum_{j_1=1}^{q!}\cdots\sum_{j_k=1}^{q!}
    \frac1N\sum_{n\leq
    N}e\left(d_n\cdot\left(h\mathbf{x}+\left(\frac{j_1}{q!},\ldots,\frac{j_k}{q!}\right)\right)\right).
  \end{multline*}
  The innermost sum is of the form
  \[\sum_{n\leq N}e(g(p_n)),\]
  with $g(x)=\sum_{i=1}^k\alpha_i\lfloor f_i(x)\rfloor$ for a certain
  $(\alpha_1,\ldots,\alpha_k)\in\R^k\setminus\Q^k$.

  By relabeling the terms we may suppose that there exists an $\ell$
  such that $\alpha_1,\ldots,\alpha_\ell\not\in\Q$ and
  $\alpha_{\ell+1},\ldots,\alpha_k\in\Q$. Furthermore we may write
  $\alpha_j=\frac{a_j}q$ for $\ell+1\leq j\leq m$. Then
  \[e(g(p_n))=e\left(\sum_{i=1}^k\alpha_k\lfloor
    f_i(p_n)\rfloor\right)= \prod_{j=1}^\ell
  s_j(\alpha_jf_j(p_n),f_j(p_n))\prod_{j=\ell+1}^kt_j(\lfloor
  f_j(p_n)\rfloor),\] where $s_j(x,y)=e(x-\{y\}\alpha_j)$ ($1\leq
  j\leq\ell$) and $t_j(z)=e\left(a_j\frac zq\right)$ ($\ell+1\leq
  j\leq k$).

  Since $s_j(x,y)$ is Riemann-integrable on $\T^2$ for
  $j=1,\ldots,\ell$ \; and $t_j(z)$ is continuous on $\Z_q=\Z/q\Z$, the
  function $\prod_{j=1}^\ell s_j\prod_{j=\ell+1}^kt_j$ is
  Riemann-integrable on $\T^{2\ell}\times\Z_q^{k-\ell}$.

  Now an application of Theorem
  \ref{baker:uniform_distribution_of_entire_functions} yields that for
  any $u\in\N$ the sequence
\[\left(\alpha_1f_1(p_n),f_1(p_n),\ldots,\alpha_\ell
  f_\ell(p_n),f_\ell(p_n),\frac{f_{\ell+1}(p_n)}u,\ldots,\frac{f_k(p_n)}u\right)_{n\geq1}\]
is u.d. in $\T^{2\ell}\times\T^{k-\ell}$. Since $\lfloor
x\rfloor\equiv a\pmod q$ is equivalent to $\frac xq\in[\frac
aq,\frac{a+1}q]$, we deduce that
\[\left(\alpha_1f_1(p_n),f_1(p_n),\ldots,\alpha_\ell
  f_\ell(p_n),f_\ell(p_n),\lfloor
  f_{\ell+1}(p_n)\rfloor,\ldots,\lfloor
  f_k(p_n)\rfloor\right)_{n\geq1}\]
is u.d. in $\T^{2\ell}\times\Z_q^{k-\ell}$, and Weyl's criterion
implies that
\[\lim_{N\to\infty}\frac1N\sum_{n\leq
  N}e\left(\sum_{i=1}^k\alpha_i\lfloor f_i(p_n)\rfloor\right)=0,\]
proving the theorem.
\end{proof}

\subsection{Functions of the form $x^\alpha\log^\beta x$}

In the one-dimensional case Boshernitzan \textit{et al.}
\cite{boshernitzan_kolesnik_quas+2005:ergodic_averaging_sequences}
showed, among other things, that these sets are vdC sets. Our aim
is to show an extended result for the $k$-dimensional case.
Therefore we use the following general criterion, which is a
combination of Fejer's theorem and van der Corput's difference
theorem.
\begin{thm}[{\cite[Theorem 3.5]{kuipers_niederreiter1974:uniform_distribution_sequences}}]
  Let $f(x)$ be a function defined for $x > 1$ that is $k$-times
  differentiable for $x > x_0$. If $f^{(k)}(x)$ tends monotonically to
  $0$ as $x\to\infty$ and if \; $\lim_{x\to\infty}x\lvert
  f^{(k)}(x)\rvert=\infty$, then the sequence $(f(n))_{n\geq1}$ is
  u.d. mod 1.
\end{thm}

Applying this theorem we get the following
\begin{cor}\label{cor:uniform_distribution_of_n_log_powers}
  Let $\alpha\neq0$ and
  \begin{itemize}
  \item either $\sigma>0$ not an integer and $\tau\in\R$ arbitrary
  \item or $\sigma>0$ an integer and $\tau\in\R\setminus[0,1]$.
  \end{itemize}
  Then the sequence $(\alpha n^\sigma\log^\tau n)_{n\geq2}$ is
  u.d. mod 1.
\end{cor}

Our third example is the following class of vdC sets.
\begin{thm}
Let $\alpha_1,\ldots,\alpha_k>0$ and $\beta_1,\ldots,\beta_k\in\R$,
such that $\beta_i\not\in[0,1]$ whenever $\alpha_i\in\Z$ for
$i=1,\ldots,k$. Then the set
\[D:=\{(\lfloor n^{\alpha_1}\log^{\beta_1}n\rfloor,\ldots,\lfloor
n^{\alpha_k}\log^{\beta_k}n\rfloor)\colon n\in\N\}\]
is a vdC set.
\end{thm}

\begin{proof}
Following the same arguments as is the proof of Theorem
\ref{thm:entire_functions_vdC} and replacing the uniform distribution
result for entire functions (Theorem \ref{thm:entire_functions_vdC})
by the corresponding result for $n^\alpha\log^\beta n$ sequences (Corollary
\ref{cor:uniform_distribution_of_n_log_powers}) yields the proof.
\end{proof}

\section*{Acknowledgment}

This research work was done when the first author was a visiting
lecturer at the Department of Analysis and Computational Number Theory
at Graz University of Technology. The author thanks the institution
for its hospitality.

The second author acknowledges support of the project F 5510-N26
within the special research area ``Quasi Monte-Carlo Methods and
applications'' founded by the Austrian Science Fund.


\providecommand{\bysame}{\leavevmode\hbox to3em{\hrulefill}\thinspace}
\providecommand{\MR}{\relax\ifhmode\unskip\space\fi MR }
\providecommand{\MRhref}[2]{%
  \href{http://www.ams.org/mathscinet-getitem?mr=#1}{#2}
}
\providecommand{\href}[2]{#2}

\end{document}